\documentclass{article}
\usepackage{amsthm,amsfonts,amsmath,amscd,amssymb,verbatim}
\usepackage{graphicx}
\usepackage{psfrag}

\usepackage{hyperref}
\usepackage{pinlabel} %added by EV.
\title{On Langmuir's periodic orbit}
\author{K.~Cieliebak, U.~Frauenfelder and M.~Schwingenheuer}
\date{}
\theoremstyle{plain}
\newtheorem{theorem}{Theorem}%[section]

\newtheorem{proposition}[theorem]{Proposition}

\theoremstyle{remark}

\newtheorem{remark}[theorem]{Remark}
\newtheorem{example}[theorem]{Example}

\newtheorem{definition}[theorem]{Definition}
\newtheorem*{definition*}{Definition}

\newcommand{\ol}{\overline}

\newcommand{\p}{\partial}

\newcommand{\wt}{\widetilde}

\newcommand{\R}{{\mathbb{R}}}

\renewcommand{\max}{{\rm max}}

\newcommand{\dd}[2]{\frac{\partial {#1}}{\partial {#2}}}

\begin{document}

\maketitle

\begin{abstract}
\noindent
Niels Bohr successfully predicted in 1913 the energy levels for the
hydrogen atom by applying certain quantization rules to classically
obtained periodic orbits. Many physicists tried to apply similar
methods to other atoms. In his well-known 1921 paper, I.~Langmuir
established numerically the existence of a periodic orbit in the
helium atom considered as a classical three body problem. In this
paper we give an analytic proof of the existence of Langmuir's periodic
orbit.   
\end{abstract}

%%%%%%%%%%%%%%%%%%%%%%%%%%%%%%%%%%%%%%%%%%%%%%%%%%%%%%%%%%%%%%%%%%%%%
\section{Introduction}\label{intro}
%%%%%%%%%%%%%%%%%%%%%%%%%%%%%%%%%%%%%%%%%%%%%%%%%%%%%%%%%%%%%%%%%%%%%

After Niels Bohr had successfully described of the spectrum of the
hydrogen atom in 1913~\cite{bohr}, leading physicists tried to apply
the same methods to the more complicated atoms like the helium
atom. Even in the seemingly easiest cases, such as predicting the
ionization potential of helium, they only obtained flawed
results. According to Langmuir~\cite{lan}, the prediction due to
Bohr's model is $28.8$ volt whereas experimental data suggests
$25.4 \pm 0.25$ volt. 
In the years to come physicists tried to think of other periodic
orbits of the two electrons in the helium atom which, after applying
the quantization rules, would make predictions that in turn could be
tested against the experimental data. Notably a model of young
Heisenberg predicted the ionization potential to be $25.6$ volt, but
Bohr rejected the idea due to the fact that Heisenberg would have
needed half integer quantum numbers. So Heisenberg abandoned the idea,
cf.~\cite{scho}. In 1921, I.~Langmuir considered a restricted form of
the classical mechanical system and found an approximate periodic
solution to the equations of motions by using a ``calculating
machine''. In his solution the electrons move simultaneously back and
forth along nearly circular arcs, situated symmetrically with respect
to an axis through the nucleus.
%being eventually repelled by each other, coming to rest and
%restarting in the opposite direction.
Applying the quantization rules to this orbit he predicted the
ionization potential to $25.62$ volt (Langmuir~\cite{lan}), in good
agreement with experimental data. 

With the advent of Schr\"odinger's and Heisenberg's quantum theory the
above-described \emph{semiclassical methods} grew out of 
fashion. There is still no exact theory of the helium atom, but the
approximations used in the modern setting are sufficiently accurate
for applications, cf.~\cite{nolting,schu}. 
With the work of M.Gutzwiller~\cite{gutz} in the 1970ies,
semiclassical methods have regained popularity, since his famous
\emph{trace formula} relates energy levels of quantum systems with
classical data (periodic orbits, their Maslov indices and
periods). Herein also lies our motivation to reconsider Langmuir's
periodic orbit in the present work. Our main result is an analytic
proof of the existence of Langmuir's periodic orbit, which we will
formulate more precisely in the next section.

%%%%%%%%%%%%%%%%%%%%%%%%%%%%%%%%%%%%%%%%%%%%%%%%%%%%%%%%%%%%%%%%%%%%%
\section{Setup}\label{setup}
%%%%%%%%%%%%%%%%%%%%%%%%%%%%%%%%%%%%%%%%%%%%%%%%%%%%%%%%%%%%%%%%%%%%%

We are going to consider the helium atom from a classical point of
view and describe the assumptions that lead to the Langmuir
problem. The nucleus of a helium atom consists of two neutrons and two
protons carrying a charge of $+2e$. Each of the two electrons in the
helium atom carries a charge of $-e$. Here $e=1.6\cdot 10^{-19} As$ is
the elementary charge, which will be set to one by rescaling in the
following. The same will be done with the electron mass. A
neutron/proton is roughly 2000 times heavier than an electron so that
the nucleus is about 8000 times heavier than each electron. We will
therefore assume the nucleus as fixed and sitting in the origin of the
coordinate system. We also restrict to the planar case, meaning that
the two electrons move in a common plane under the attractive force of
the nucleus and their mutually repelling force. We thus
consider a variant of the three body problem of celestial
mechanics where the force between the two lighter masses is repelling
rather than attracting. Note that the influence of none of the three
bodies on the other two is negligeable, so this is not a variant of
the well-studied ``restricted'' three body problem.
%What makes the problem even harder in general is the fact
%that the ``satellite'' here (one of the electrons) has no neglectable
%mass (here charge).
%One therefore often tries to find invariant
%subspaces of the full phase space. A trajectory starting in an
%invariant subspace is confined to it for all times, reducing the
%complexity of the problem.

Each electron is described by its two position coordinates
$q_i=(q_i^x,q_i^y)$ and its two momentum coordinates $p_i=(p_i^x,p_i^y)$,
$i=1,2$. Therefore the full phase space is eight dimensional and the
Hamiltonian of the full problem governing the dynamics is given by 
$$
   H(q_1,q_2,p_1,p_2)=\frac{1}{2}\left(|p_1|^2+|p_2|^2\right)-\frac{2}{|q_1|}-\frac{2}{|q_2|}+\frac{1}{|q_1-q_2|}.
$$
Here the first term describes the kinetic energy of the electrons,
the second and third terms describe the Coulomb attraction by the
nucleus, and the last term describes the Coulomb repulsion between the
two electrons. 

\subsection{The symmetry and the Langmuir Hamiltonian}

%The phase space associated to a mechanical problem is an instance of a
%symplectic manifold. In the 19th century, Liouville observed that a
%region in phase space keeps its volume (being nevertheless
%stretched and twisted) when the mechanical system evolves over time,
%but only in 1965 did V.I.~Arnold describe the invariance of a finer
%structure, the symplectic form. A symplectic form $\omega$ on a
%manifold $M$ is a  closed, nondegenerate 2-form. In the case of $M$
%being the phase space of a mechanical system, there is a canonical
%choice for the symplectic form: $\omega=\Sigma dq_i \wedge dp_i$.
The phase space of a mechanical system carries a canonical
symplectic form $\omega=\Sigma dp_i \wedge dq_i$ which is preserved
under the time evolution of the system. This observation is already
implicit in the work of J.-L.~Lagrange and S.~Poisson
(see~\cite{Souriau},~\cite{Marle}), and explicit in the work of H.~Poincar\'e and
\'E.~Cartan.\footnote{We thank C.~Viterbo for setting the
  history straight.}  
In today's terminology, a symplectic form $\omega$ on a manifold $M$
is a closed, nondegenerate 2-form. A diffeomorphism $\phi$ of $M$ is called a symplectomorphism (also
called a canonical transformation in physics) if
it preserves the symplectic form: $\phi^\ast \omega=\omega$. To any
smooth function $H\colon M \to \R$ we can associate a one-parameter
family of symplectomorphisms $\phi_t^H$ by integrating the Hamiltonian
vector field $X_H$, which is uniquely defined by the nondegeneracy of
$\omega$ via
$$
   \iota_{X_H} \omega=-dH.
$$ 
If $H$ does not depend on time we can think of $H$ as the total
energy of the mechanical system. A point $z$ in phase space
constitutes the initial conditions of the mechanical system, and
$\phi_t^H(z)$ describes how the mechanical system evolves over
time.

A short calculation shows that for any symplectomorphism $\phi$ and
Hamiltonian vector field $X_H$, the vector field $\phi_\ast X_H$ is
associated to the Hamiltonian function $H\circ \phi^{-1}$. 
Assume now that $\phi$ is a symplectomorphism which preserves $H$:
$H\circ \phi^{-1}=H$. Then we readily see that the vector fields
associated to $H$ and $H\circ \phi^{-1}$ coincide: $\phi_\ast
X_H=X_H$. This means that $\phi$ preserves the Hamiltonian flow,
i.e. the time evolution of the mechanical system.  
Assume further that $\phi$ is a symplectic involution, meaning a
symplectomorphism with $\phi^2=id$, which preserves the Hamiltonian
function $H$. Of particular interest is the set $F$ of fixed points of
$\phi$. Since $\phi$ is an involution, one can show that $F$ is a
smooth manifold such that $X_H$ is tangent to $F$. Accordingly, $F$ is
preserved under the Hamiltonian flow. Thus finding a symplectic
involution which preserves $H$ leads to a dynamical system of lower
complexity, regarding $F$ instead of the full phase space. 

The symplectic involution which yields the Langmuir problem is given
by
$$
   \tau \colon \R^8 \to \R^8,\qquad (q_1,q_2,p_1,p_2) \mapsto
   (\overline{q}_2,\overline{q}_1,\overline{p}_2,\overline{p}_1).
$$
Here $\overline{q}$ means complex conjugation where the vector
$q=q^x+iq^y$ is written in complex notation. A short calculation shows
that $\tau$ is symplectic for the standard symplectic form and leaves
$H$ invariant. Its fixed point set is $F=\{q_1=\overline{q_2},\;
p_1=\overline{p}_2\}$. In terms of the variables $q=q_1$ and $p=p_1$
on $F$ this leads to the \emph{Langmuir Hamiltonian} on $\R^4$, 
$$
   H(q,p)=|p|^2-\frac{2}{|q|}-\frac{2}{|q|}+\frac{1}{|q-\overline{q}|}=|p|^2-\frac{4}{|q|}+\frac{1}{2|\textrm{Im}(q)|}.
$$
What this amounts to from the physical point of view is that $\tau$
interchanges the electrons but at the same time reflects them in the real
axis. Thus the fixed point set F consists of pairs of electrons which
perform a mirrored movement. The system has angular momentum zero at
all times because the angular momenta of the two electrons cancel.
In Cartesian coordinates $q=(x,y)$ and $p=(p^x,p^y)$ the Hamiltonian reads
$$
   H(x,y,p)=|p|^2-\frac{4}{\sqrt{x^2+y^2}}+\frac{1}{2|y|}.
$$
Note that $y=0$ corresponds to collisions of the two electrons. Since
we are interested in orbits without collisions, we will restrict our
attention to the region where $y>0$ on which the Hamiltonian simplifies to
$$
   H(x,y,p)=|p|^2-\frac{4}{\sqrt{x^2+y^2}}+\frac{1}{2y}.
$$
We will call $V(x,y)=-\frac{4}{\sqrt{x^2+y^2}}+\frac{1}{2y}$ the \emph{Langmuir potential}.
The equations of motion $\dot{q}=2p, \ \dot{p}=-\nabla V$ are thus given by
\begin{equation}\label{eq:Ham}
\left\{
\begin{aligned}
  \ddot{x} &= \frac{-8x}{(x^2+y^2)^\frac{3}{2}}, \\
  \ddot{y} &= \frac{-8y}{(x^2+y^2)^\frac{3}{2}}+\frac{1}{y^2}.
\end{aligned}
\right.
\end{equation}

\subsection{Hill's regions for the Langmuir potential}\label{equipot}

We are interested in solutions of negative energy $E<0$. For such
solutions the coordinates $(x,y)$ are confined to the {\em Hill's region}
$$
   \mathcal{H}_E := \{(x,y)\mid V(x,y)\leq E\}.
$$
Its boundary is given by the equipotential line $\{V=E\}$ and
corresponds to points with zero velocity. 
Figure~\ref{fig:2} shows a 3D plot of the Langmuir potential with
the equipotenial line for energy $E=-1$. 
\begin{figure}[h!]
 \includegraphics[width=14cm]{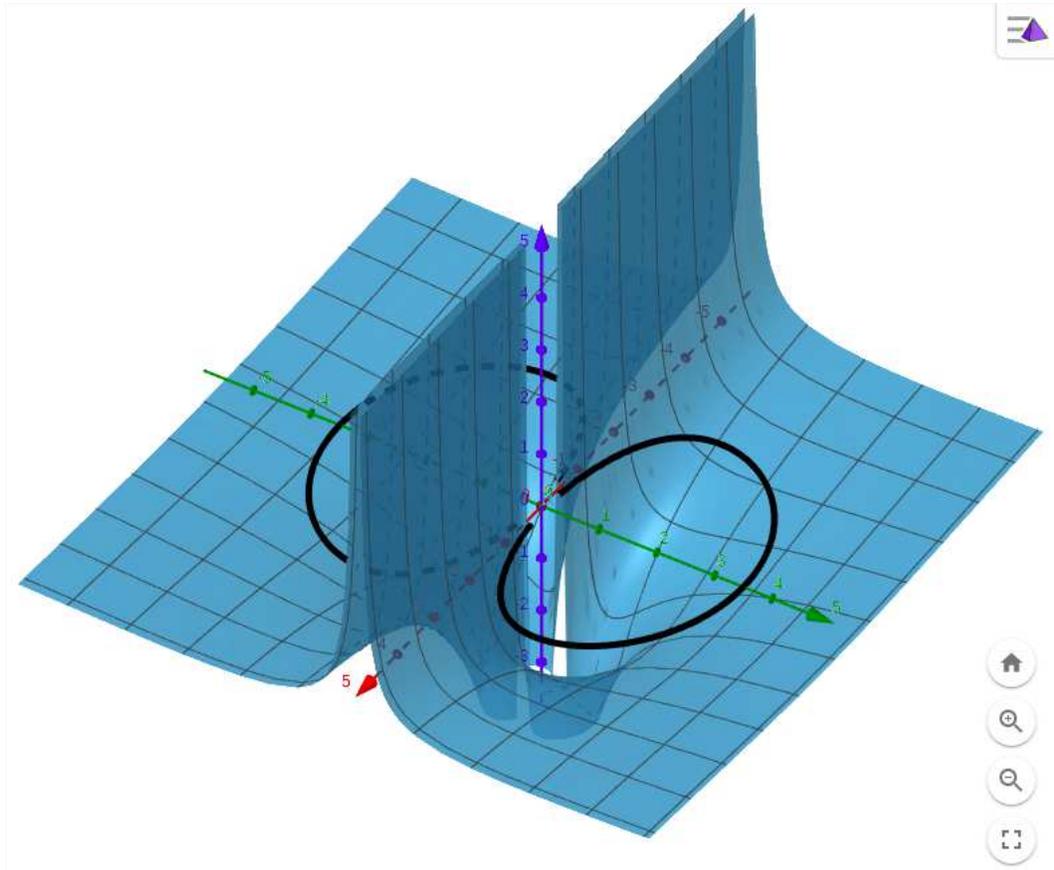}
 % Langmuirequipot.eps: 2458x2058 px, 300dpi, 20.81x17.42 cm, bb=0 82 590 576
 \caption{Langmuir potential and equipotential line}
 \label{fig:2}
\end{figure}
%The black kidney-shaped curve is drawn on height zero, but it
%actually lies on the surface on height $E=-1$. So an electron with
%vanishing velocity and total energy $E=-1$ has to be somewhere on the
%black line. If we however increase its kinetic energy, by
%conservation of energy, we have to decrease its potential energy so
%that the electron has to be inside of the black curve when projected
%to the x,y-plane. In accordance with the three body problem of
%celestial mechanics we refer to the inside of the black curve when
%projected to the x,y-plane as the Hills region of the electron for
%energy $E=-1$.
Note that the Hill's region for $E=-1$ contains the interval
$\left[0,\frac{7}{2}\right]$ on the y-axis, and solutions starting on
the $y$-axis with zero velocity fall into the origin in finite time.
The Hill's regions are bounded for $E<0$, but they become unbounded at
the ionization energy $E=0$. A short calculation shows that Hill's
region at $E=0$ is bounded by the two lines $V^{-1}(0)=\{(x,y)\mid
y=\frac{1}{\sqrt{63}}|x| \}$.

\subsection{Scaling invariance of the Langmuir Hamiltonian}\label{scale}

A salient feature of the Langmuir potential is its homogeneity of
degree $-1$: $V(aq)=a^{-1}V(q)$ for $a>0$. The diffeomorphisms 
$$
   \beta_a(q,p) := (aq,\frac{1}{\sqrt{a}}p)
$$
thus satisfy $H\circ\beta=a^{-1}H$. Moreover, they are conformally
symplectic: $\beta^\ast \omega =\sqrt{a}  \omega$. Now a short
computation yields: \emph{If $z(t)$ is a solution of~\eqref{eq:Ham} of
  energy $E$, then $\wt z(t)=\beta_a\bigl(z(a^{-3/2}t)\bigr)$ is again
  a solution of~\eqref{eq:Ham} of energy $a^{-1}E$.} 

%A priori the dynamics on the energy hypersurface $H^{-1}(E)$ for different E can change dramatically. One feature of the Langmuir problem is that this does not happen. To see this, we consider new (scaled) coordinates: $x=a\tilde{x}; y=a\tilde{y}; \ p=b\tilde{p}$ for $a,b >0$ real numbers which we insert in the Langmuir Hamiltonian for total energy E:
%$$|b\tilde{p}|^2-\frac{4}{a\sqrt{\tilde{x}^2+\tilde{y}^2}}+\frac{1}{2a\tilde{y}}-E=0.$$ Now we divide by $b^2$ to obtain: $$|\tilde{p}|^2-\frac{4}{ab^2\sqrt{\tilde{x}^2+\tilde{y}^2}}+\frac{1}{2ab^2\tilde{y}}-\frac{E}{b^2}=0.$$
%If we choose $b=\frac{1}{\sqrt{a}}$ then $ab^2=1$ and we obtain:
%$$|\tilde{p}|^2-\frac{4}{\sqrt{\tilde{x}^2+\tilde{y}^2}}+\frac{1}{2\tilde{y}}-a\cdot E=0.$$
%Since the mapping (change of coordinates)
%$\beta(q,p)=(aq,\frac{1}{\sqrt{a}}p)$ is conformally symplectic, i.e. $\beta^\ast \omega =\sqrt{a}  \omega$, the Hamiltonian vector field is also only scaled so that the trajectories remain invariant and are only traversed at different speeds. Hence the dynamics on the energy hypersurface for E is just given by the image under $\beta$ of the dynamics on the energy hypersurface for $aE$. In that sense the dynamics is invariant and we can without loss of generality set the total energy once and for all equal to $E=-1$ (except for notice to the contrary!).
So the dynamics on energy hypersurfaces $H^{-1}(E)$ for different
$E<0$ differ only by their time parametrization. In the following
discussion we will therefore often restrict our attention to the case
$E=-1$. 

\subsection{The magical line}
An important player in the sequel will be the set in the x-y-plane
where the attractive force on the electron from the nucleus and the
repelling force of the other electron in vertical direction cancel.
Setting $\ddot y=0$ in~\eqref{eq:Ham}, we find that this set is the
pair of lines
$$
   \sqrt{3}y=|x|.
$$
We will refer to it as the \emph{magical line}.
%Indeed from the equation of motions for the vertical force we obtain:
%$$\ddot{y}=-\frac{8y}{(x^2+y^2)^{\frac{3}{2}}}+\frac{1}{y^2}.$$ Setting this equal to zero we obtain $$ 8y^3=(x^2+y^2)^{\frac{3}{2}} \Rightarrow 2y=(x^2+y^2)^{\frac{1}{2}} \Rightarrow 4y^2=x^2+y^2.$$
%Thus $$3y^2=x^2 \Rightarrow \sqrt{3}y=|x|$$ is the equation of the magical line.
Above the magical line we thus have $\ddot y<0$, while below it we have
$\ddot y>0$.

\subsection{Langmuir orbits and the main result}

Let us fix an energy $E\leq 0$ and a height $h>0$ satisfying
$\frac{7}{2h}+E\geq 0$, so that the point $(0,h)$ lies in the Hill's
region at energy $E$. We are interested in solutions of~\eqref{eq:Ham}
of energy $E$ that start at the point $(0,h)$ on the $y$-axis in
horizontal direction to the right, i.e., they satisfy the initial conditions
$$
   x(0)=0,\quad y(0)=h,\quad
   \dot{x}(0)=2|p|=2\sqrt{\frac{7}{2h}+E},\quad \dot{y}(0)=0. 
$$ 
We will refer to this initial value problem as the \emph{Langmuir
  problem} at energy $E$ and height $h$. 
Figure~\ref{fig:Langmuir-problem} shows solutions of the Langmuir
problem for energy $E=-1$ and two different heights. It also shows the
boundary of the Hills region and the magical line. 
\begin{figure} 
\begin{center}
\includegraphics[width=12cm]{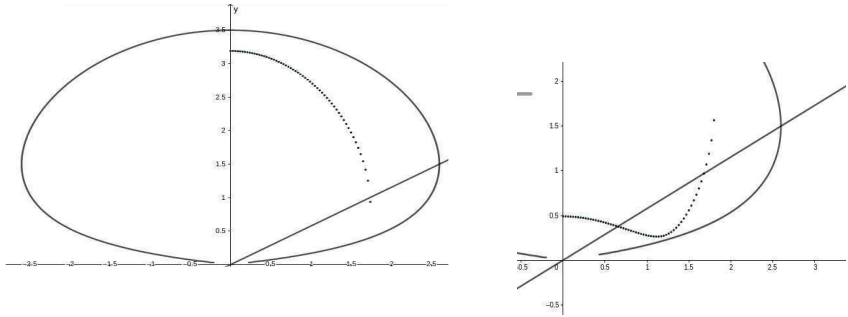}
% numericpic.eps: 2500x1383 px, 300dpi, 21.17x11.71 cm, bb=0 514 600 846

\end{center}
 \caption{Solutions of the Langmuir problem for different initial heights}
 \label{fig:Langmuir-problem}
\end{figure}
We now define the protagonist of this paper:
\begin{definition*}
A \emph{Langmuir orbit} of energy $E$ is a solution to
the Langmuir problem (with some $h$) whose velocity vector vanishes at
some time $T>0$.
\end{definition*}

Thus a Langmuir orbit touches the boundary of the Hill's region at
time $T$. After that it reverses its direction and travels back along
the same trajectory, hitting again the point $(0,h)$ at time $2T$,
performs the same motion in the negative $x$-direction, and then
repeats itself with period $4T$. This is the periodic orbit described
by Langmuir in~\cite{lan}. Therefore, the main result of this paper
can be phrased as

\begin{theorem}\label{langmuir}
For each negative energy $E<0$ there exists a Langmuir orbit.
\end{theorem}

Let us mention that in~\cite{diacu-perez-chavela}, F.~Diacu and
E.~P\'erez-Chavela claim the existence of infinitely many periodic
orbits for a system consisting of $n$ electrons
situated at the vertices of a regular $n$-gon whose size changes
homothetically, and a nucleus moving along an orthogonal line through
the center of the $n$-gon.
%While their result also holds for $n=2$, in
%that case all their periodic orbits are collision orbits and therefore
%do {\em not} correspond to a Langmuir orbit. 
For $n=2$, their problem is planar and for a suitable choice of parameters becomes
mathematically equivalent to the 
system~\eqref{eq:Ham}. Unfortunately, we were not able to follow their
arguments. In particular, their Theorem\,5 excludes the existence of
a Langmuir orbit (which is ``equally symmetric'' in their
terminology), contradicting our Theorem~\ref{langmuir}. 

Theorem~\ref{langmuir} will be proved in Section~\ref{mainpart}. Its proof
will require an understanding of the Langmuir problem at energy zero, 
which is the content of the next subsection.
%In particular, we will show (STEP 4 of the proof of Theorem~\ref{langmuir}) that the assumed non-existence of a periodic Langmuir orbit would result either in a contradiction of Proposition ~\ref{PropLimitorbit} or that the trajectory cannot be confined to the energy zero Hill's region. This contradiction eventually shows the absurdity of the assumption.
%Therefore we first need a condition of the trajectory in the energy zero Hill's region which is accessible to later considerations. This condition will be formulated (and proved) in Propostion  ~\ref{PropLimitorbit} below. 

\subsection{The Langmuir problem at energy zero}\label{zero}

In this subsection we consider the Langmuir problem at energy
$E=0$. By the rescaling argument in Section~\ref{scale}, it suffices
to consider the case $h=1$. For a curve $(x(t),y(t))$ in the plane we
denote by $(r(t),\phi(t))$ the corresponding curve in polar coordinates.
The following proposition will be used in STEP 4 of the proof of
Theorem~\ref{langmuir}. Note that it shows in particular that the
solution to the Langmuir problem at energy $E=0$ exists for all times $t\in[0,\infty)$. 

\begin{proposition}\label{PropLimitorbit}
The solution $(x(t),y(t))$ to the Langmuir problem at energy $E=0$ and
height $h=1$ satisfies $\dot r(t)>0$ for all $t\in (0,\infty)$. 
\end{proposition}

\begin{proof}
We first perform a circle inversion in configuration space $q\mapsto
\frac{1}{\overline{q}}$ to get equations that are easier to manipulate. In
order to get an equivalent dynamical problem we have to perform a
symplectic transformation on phase space, so we have to
transform the momentum variable by $p \mapsto -q^2\ol p$. In the
transformed variable the Langmuir Hamiltonian reads
$$
   H=|q|^4|p|^2-4|q|+\frac{|q|^2}{2\textrm{Im}(q)}.
$$
Dividing by $|q|^4$ yields the Hamiltonian
$$
   \wt{H}=|p|^2-\frac{4}{|q|^3}+\frac{1}{2|q|^2\textrm{Im}(q)}
$$
whose Langmuir problem at energy $0$ and height $1$ corresponds under
the inversion, up to time reparametrization, to the original Langmuir
problem at energy $0$ and height $1$. Thus showing $\dot r>0$ for the
original problem is equivalent to showing $\dot r<0$ for the Langmuir
problem with Hamiltonian $\wt H$.  
%This of course changes the dynamics but only by reparametrisation
%(comparing to the two observations above: $\lambda(t)=|q(t)|^4$ is
%positiv, so that the reparametrisation is monotonically
%increasing). Therefore a trajectory which falls into the origin in
%this problem corresponds to a trajectory which goes off to infinity
%in the original problem.

To proceed, we rewrite the Hamiltonian $\wt H$ in polar coordinates
$(r,\phi)$ as
$$
   \wt{H}(r,\phi,p_r,p_\phi)=p^2_r+\frac{p^2_\phi}{r^2}-\frac{4}{r^3}+\frac{1}{2r^3
     \sin{\phi}},
$$
where $p_\phi,p_r$ are the conjugate momentum variables. Note that the
potential
$$
   \tilde{V}(r,\phi) = -\frac{4}{r^3}+\frac{1}{2r^3 \sin(\phi)}
$$
is homogeneous with respect to the $r$-variable of degree $-3$.
In general, for a homogeneous potential of degree $\alpha$ we get from
the chain rule: $\dd{V}{r}=\frac{\alpha}{r}V(r,\phi)$. So
Hamilton's equations (those which are relevant to the present discussion) become 
\begin{align*}
   \dot{r} &= \dd{\tilde{H}}{p_r}=2p_r,\cr
   \dot{p_r} &=
   -\dd{\tilde{H}}{r}=-\dd{V}{r}-\dd{}{r}\left(p^2_r+\frac{p^2_\phi}{r^2}\right)
   = -\frac{\alpha}{r}V(r,\phi)+2\frac{p^2_\phi}{r^3}.
\end{align*}
But for energy $E=0$ we have $-\wt V(r,\phi)=p^2_r+\frac{p^2_\phi}{r^2}$, hence
$$
   \dot{p_r}=\frac{\alpha \cdot p^2_r}{r}+ \frac{\alpha \cdot
     p^2_\phi}{r^3} +2\frac{p^2_\phi }{r^3}
   = \frac{1}{r}\left(\alpha\cdot p^2_r+(\alpha+2)\cdot
   \frac{p^2_\phi}{r^2} \right).
$$ 
Combining this with the equation for $\dot{r}$, we get
$$
   \ddot{r}=\frac{2}{r}\cdot\left(\alpha\cdot p^2_r+(\alpha+2)\cdot
   \frac{p^2_\phi}{r^2}\right).
$$
In our case we have $\alpha=-3$, so the preceding equation shows
$\ddot{r}<0$. Since the Langmuir solution for $\wt H$ starts with
$\dot{r}(0)=0$, this implies $\dot{r}(t)<0$ for all $t>0$. Under the
circle inversion this corresponds to $\dot{r}(t)>0$ in the original
Langmuir problem, so the proposition is proved.
\end{proof}

%%%%%%%%%%%%%%%%%%%%%%%%%%%%%%%%%%%%%%%%%%%%%%%%%%%%%%%%%%%%%%%%%%%%%
\section{Existence of a Langmuir orbit}\label{mainpart}
%%%%%%%%%%%%%%%%%%%%%%%%%%%%%%%%%%%%%%%%%%%%%%%%%%%%%%%%%%%%%%%%%%%%%

%\begin{theorem}\label{langmuir}
%The Langmuir ODE possesses a Langmuir orbit.
%\end{theorem}
\subsection{Outline of the proof}
In this section we prove Theorem~\ref{langmuir} in four Steps.
By the rescaling argument in Section~\ref{scale} we can and will restrict to the case
of energy $E=-1$ in Steps 1--3. 
Then the height $h$ in the Langmuir problem varies
in the interval $\left(0,\frac{7}{2}\right)$ and we denote the
corresponding solution by $(x_h,y_h)$. 
Since any solution is confined to the compact Hill's region for $E=-1$, any trajectory of the Langmuir problem for this energy has to come to rest in x-direction. We therefore place our argument on the time of first vanishing of the velocity-component in x-direction:
for each $h\in\left(0,\frac{7}{2}\right)$ we define
$$
   t_h := \inf\{t>0\mid \dot{x}_h(t)=0 \} \in (0,\infty).
$$
\smallskip
\textbf{STEP 1: The map $h\mapsto t_h$ is smooth.} 

Therefore, we can define a smooth map
$$
   \alpha \colon \left(0,\frac{7}{2}\right) \to \R,\qquad \ h \mapsto
   \dot{y}_h(t_h).
$$
Note that an $h$ with $\alpha(h)=0$ corresponds to a Langmuir orbit. 
Now Figure~\ref{fig:Langmuir-problem} suggests that $\alpha(h)$ should be
negative for $h$ close to $7/2$ and positive for $h$ close to $0$, so
in between it should have a zero (the desired Langmuir orbit).
\medskip

\textbf{STEP 2: There exists $\delta>0$ such that $\dot
y_h(t)<0$ for all $t\in(0,t_h]$ and $h\in[7/2-\delta,7/2)$.}

In particular, this implies $\alpha(h)<0$ for $h\in[7/2-\delta,7/2)$.
Rather than proving positivity of $\alpha$ for $h$ near $0$ directly,
we will argue by contradiction and {\bf assume that there exists no Langmuir
orbit.} Consequently, the continuity of $\alpha$ and STEP 2 imply $\alpha(h)<0$
for all $h\in\left(0,\frac{7}{2}\right)$, since $\alpha$ cannot have any zeros by that assumption.
\medskip

\textbf{STEP 3: The assumption implies $\dot{y}_h(t)<0$ for all $t\in(0,t_h]$ and all $h\in(0,\frac{7}{2})$.} \\

We again argue by contradiction: If the assertion was not true, then
there exists an $h_1\in(0,\frac{7}{2}-\delta)$ and a $t_1 \in (0,t_{h_1})$ such that
$\dot{y}_{h_1}(t_1)=0$. But then:
%for the derivatives $\dot y_h$ we have:  

\begin{minipage}{5.8cm}
\psfrag{s}{$\tilde{h}=\frac{7}{2}-\delta$}
\psfrag{k}{$h_0$}
\psfrag{q}{$t_1$}
\psfrag{p}{$t_0$}
\psfrag{h}{$h_1$}
\psfrag{y}{$\dot{y}_h$}
\psfrag{t}{$t$}
 \includegraphics[width=5cm]{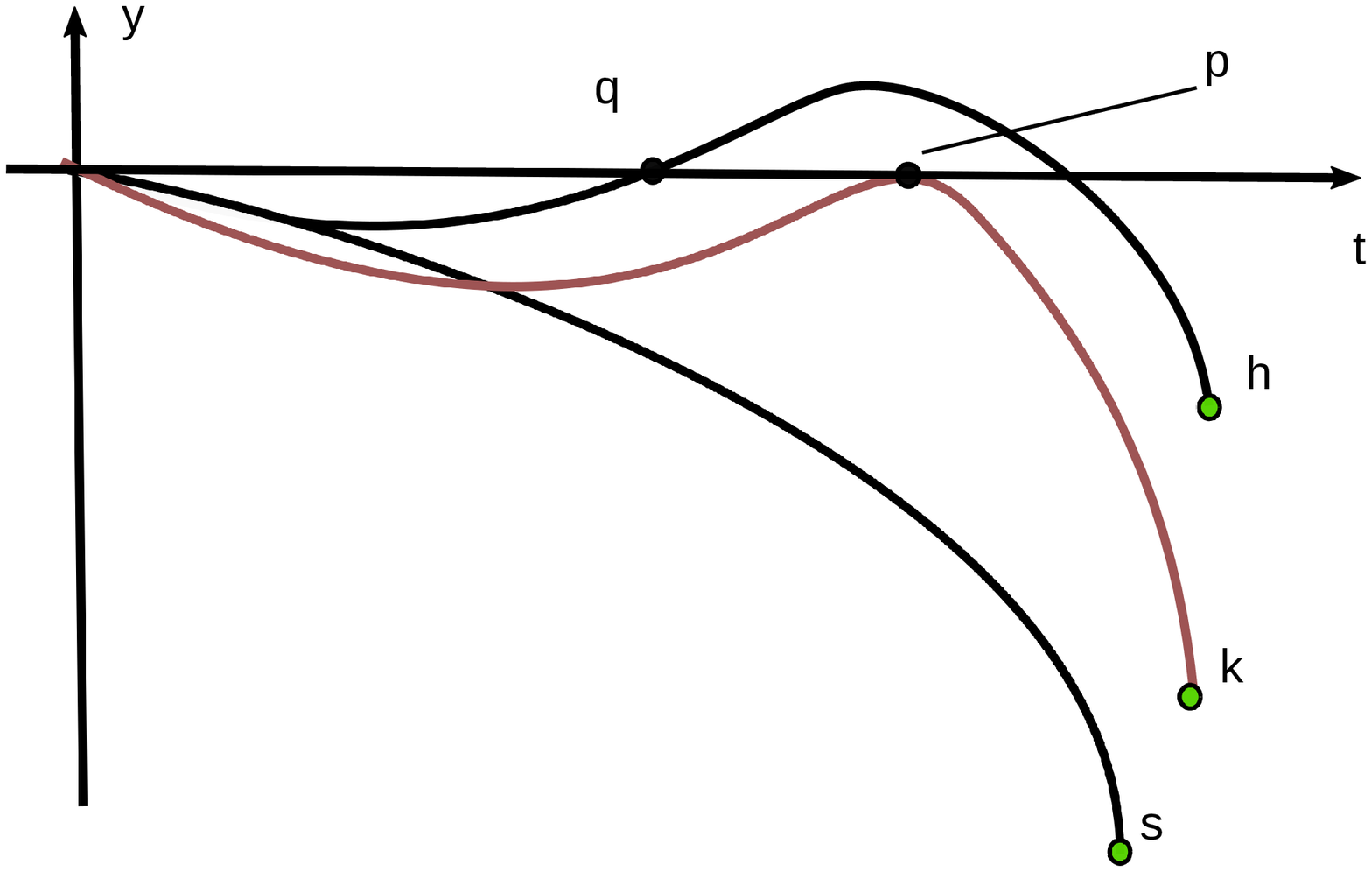}
 % step2-2.eps: 2191x1383 px, 300dpi, 18.55x11.71 cm, bb=30 486 556 818
\end{minipage}\begin{minipage}{6cm} Changing from a strictly negative
  derivative for $\tilde{h}=\frac{7}{2}-\delta$ to one with a zero at
  $t_1$ for $h_1$, we have to go through an $h_0$ with a $t_0$ at which
  $\dot y_{h_0}(t_0)=0$ and $\ddot{y}_{h_0}(t_0)=0$. Note that due to
  $\alpha(h)<0$ all ends are bound to be negative. 
\end{minipage} 

But points with vanishing force in the $y$-direction lie on the magical
line, thus $y_{h_0}(t)$ crosses the magical line horizontally at some
$t_0$ for the first time.  
But the force above the magical line is negative in the $y$-direction,
which with $\dot{y}_{h_0}(0)=0$ implies $\dot{y}_{h_0}(t_0)<0$, the desired contradiction. 
\medskip

\textbf{STEP 4: The conclusion of STEP 3 contradicts Proposition ~\ref{PropLimitorbit}.}\\ 
To prove this Step, we change our point of view: Instead of varying h,
we fix $h=1$ and vary the energy E (cf. Section \ref{scale}), so that
in the limit $E \to 0$ Proposition~\ref{PropLimitorbit} becomes applicable. 
The idea is roughly as follows: Since the orbit at zero energy (the
limit orbit in the following) has strictly increasing distance from
the origin, continuity forces the orbits closer and closer to this
limit orbit to stay longer and longer in the vincinity of the limit
orbit before falling back. This is not compatible with the conclusion
of STEP 3 that the velocity in $y$-direction is negative, and the
resulting contradiction finishes the proof. 

%We argue by contradiction, therefore we assume the contrary, that there exists no periodic orbit with vanishing velocity vector. The Hills region comprises the part of the y-axis between $\left[0,\frac{7}{2}\right]$.
%We consider solutions to the Langmuir ODE (see above) and define:
%$$t_h=min\{t>0: \dot{x}_h(t)=0 \}.$$ The intuition behind that definition is to encode the time when the horizontally shot particle is decelerated down to zero in x-direction and starts moving in the opposite direction. We are going to get the contradiction in 3 Steps.\newline

\subsection{STEP 1: The map $h \mapsto t_h$ is smooth.}
We begin by showing that the map $h\mapsto t_h$ is well-defined and smooth.
Hill's region at energy $E=-1$ is described by the inequality
$$
   1 + \frac{1}{2y} \leq \frac{4}{\sqrt{x^2+y^2}}.
$$
This first implies $1 + \frac{1}{2y} \leq \frac{4}{y}$, and therefore
$y\leq7/2$. Next, it implies $1 + \frac{1}{2y} \leq \frac{4}{|x|}$,
which together with the bound on $y$ yields $|x|\leq7/2$. Inserting
these estimates into the first equation in~\eqref{eq:Ham}, we obtain 
$$
   \ddot{x} = \frac{-8x}{(x^2+y^2)^{3/2}} \leq -8\gamma x
   \quad\text{with}\quad \gamma := \frac{1}{(\frac{49}{4}+\frac{49}{4})^{3/2}}.
$$
We interpret this as saying that the actual force in negative
x-direction ($\ddot{x}$) is at least as strong as $-8\gamma x$. Thus a
particle being shot in $x$-direction subject to the Langmuir
Hamiltonian will come to rest no later than a particle being decelerated
with force $-8\gamma x$. Hence solving $\ddot{x}=-8\gamma x$ gives an
upper bound for the time $t_h$. The solution to the latter ODE is
$x(t)=A\sin(\omega t)$ with $\omega=\sqrt{8\gamma}$. So
$\dot{x}(t)=A\omega \cos(\omega t)$ vanishes for
the first time if $\omega t=\frac{\pi}{2}$ and we obtain the estimate
$$
   t_h\leq t_\max := \frac{\pi}{2\sqrt{8\gamma}}.  
$$

Now let us denote by $t_L(h)$ the {\em lifetime} of the electron in the
Langmuir problem at energy $-1$ and height $h$. Assume first that $t_L(h)<t_\max$.
Since the orbit starts out with positive velocity in the $x$-direction, the $x$-velocity has to
vanish at some time before it falls into the origin, so we have
$t_h<t_L(h)<t_{max}$. If however $t_L(h)\geq t_\max$, then $t_h$ is
well-defined and finite by the inequality above. This proves that
$t_h\in(0,\infty)$ for all $h\in(0,7/2)$ and that the map $h\mapsto
t_h$ is well-defined. 
To show that it is smooth, we consider the open set
$$
   W := \bigcup_{h\in(0,\frac{7}{2})} \{h\}\times (0,t_L(h))\subset
   (0,\frac{7}{2})\times (0,\infty)
$$
and the smooth function
$$
   F\colon W\to\R,\qquad (h,t)\mapsto\dot x_h(t). 
$$
%To see that $W$ is open is equivalent to notice that the function $h \to t_L(h)$ is lower semi-continues.
%This is interesting in its own right, but we will not elaborate further on this fact.
%Actually, let $(h,t)\in W$. By construction $t<t_L(h)$, and the electron must be away from the origin. Thus $y_h(t)=\epsilon>0$. Now by continuity of $y_h(t)$ at $(t,h)$, choose $\delta>0$ such that for all $(h',t')$ with $|(h,t)-(h',t')|<\delta$ (in standard euclidean norm) we have $$|y_h(t)-y_{h'}(t')|<\epsilon/3.$$ This implies that $y_{h'}(t')>2\epsilon/3>0$, in particular is $t'<t_L(h')$. Hence the open ball $B_{\delta}((h,t))\subset W$, showing that $W$ is open.
(Openness of $W$ is equivalent to lower semicontinuity of the map
$h\mapsto t_L(h)$, which in turn follows directly from continuous
dependence of solutions of the ODE on the initial conditions). 
Since $\frac{\p F}{\p t}(h,t)=\ddot x_h(t)<0$ by~\eqref{eq:Ham} for
all $(h,t)$, it follows from the implicit function theorem that
$F^{-1}(0)$ is the graph of a smooth function $h\mapsto t_h$. 
This proves that the map $h\mapsto t_h$, and therefore also the map
$\alpha:h\mapsto\dot y_h(t_h)$, is well-defined and smooth. 

\subsection{STEP 2: There exists $\delta>0$ such that $\dot
y_h(t)<0$ for all $t\in(0,t_h]$ and $h\in[7/2-\delta,7/2)$.}
%We begin by deriving a uniform lower bound on $t_h$ for $h\geq 3$. For this,
%we estimate the second equation in~\eqref{eq:Ham} by setting $x=0$:
%$$
%  \ddot{y} = \frac{-8y}{(x^2+y^2)^\frac{3}{2}}+\frac{1}{y^2} \geq
%  \frac{-8y}{(y^2)^\frac{3}{2}}+\frac{1}{y^2} = -\frac{7}{y^2}.
%$$
%Let $T_1$ be the smallest time in which a solution of the ODE $\ddot
%y=-7/y^2$ with initial conditions $\dot y(0)=0$ and $y(0)=h\in[3,7/2]$
%reaches $1$. Then by the preceding estimate the solution $(x_h,y_h)$
%of the Langmuir problem satisfies $y_h(t)\geq 1$ for all $t\leq T_1$.
%Now we estimate the first equation in~\eqref{eq:Ham}:
%$$
%  \ddot{x} = \frac{-8x}{(x^2+y^2)^\frac{3}{2}} \geq \frac{-8x}{y^3}
%  \geq -8x \quad\text{for}\quad t\leq T_1.
%$$

To see this, note first that at time $t=0$ we have $x_h(0)=0$ and
$y_h(0)=h$, so the second equation in~\eqref{eq:Ham} yields $\ddot y_h(0) = -7/h^2<0$.
By continuity there exists $\tau>0$ such that for each $h\in[3,7/2]$
we have $\ddot y_h(t) <0$ for all $t\in[0,\tau]$, which in view of
$\dot y_h(0)=0$ implies $\dot y_h(t)<0$ for all $t\in(0,\tau]$.  

Now we argue by contradiction and assume there exist sequences
$h_n\nearrow 7/2$ and $t_n\in(0,t_{h_n}]$ with $\dot y_{h_n}(t_n)=0$. 
By the preceding estimate we must have $t_n\geq\tau$ for large $n$,
and by the upper bound from STEP 1 we have $t_n\leq t_\max$. So a
subsequence of $t_n$ converges to some $t_*\in[\tau,t_\max]$ such that
$\dot y_{7/2}(t_*)= 0$. By construction we have $t_*\leq t_L(7/2)$, the
lifetime of the solution for $h=7/2$. But for $h=7/2$ the solution of
the Langmuir problem falls straight into the origin along the $y$-axis,
hence $\dot y_{7/2}(t)<0$ for all $0<t< t_L(7/2)$, which yields the
desired contradiction provided that $t_*<t_L(7/2)$.

It remains to show that $t_*\neq t_L(7/2)$. 
Again, we assume the contrary and note that this would imply that
$(x_{h_n}(t_n),y_{h_n}(t_n))$ converges to $(0,0)$ because, as was
mentioned above, the trajectory for $h=7/2$ falls straight into the
singularity at $(0,0)$. This and the assumption that
$\dot{y}_{h_n}(t_n)=0$ together with the energy relation imply
that $|\dot{x}_{h_n}(t_n)|\to \infty$, as only
$\dot{x}_{h_n}(t_n)^2$ contributes to the kinetic energy whereas the
potential energy tends to $-\infty$.
%So choose $N_0$ so big that $$2<|\dot{x}_{h_n}(t_n)|,$$ for $n\geq N_0$. 
On the other hand, $t_n \leq t_{h_n}$ and equation~\eqref{eq:Ham}
imply that $\ddot x(t)\leq 0$ for all $t\in[0,t_n]$, thus 
$$|\dot{x}_{h_n}(t_n)|\leq\dot{x}_{h_n}(0)=2\cdot\sqrt{\frac{7}{2h_n}-1}.$$ 
%So choose $n\geq N_0$ so big that
%$$2<|\dot{x}_{h_n}(t_n)|<\sqrt{\frac{7}{2h_n}-1}<\frac{1}{2}.$$
But as $n\to\infty$ the right hand side tends to $0$ because $h_n\to
7/2$, which contradicts $|\dot{x}_{h_n}(t_n)|\to \infty$ and finishes
the proof of STEP 2. 

\medskip

From now on we argue by contradiction and {\bf assume that there
  exists no Langmuir orbit}. By continuity of $\alpha$ and STEP 2 this
implies $\alpha(h)<0$ for all $h\in\left(0,\frac{7}{2}\right)$. 

\subsection{STEP 3: The assumption implies $\dot{y}_h(t)<0$ for
all $t\in(0,t_h]$ and all $h\in\left(0,\frac{7}{2}\right)$.} 

To see this, first note that by assumption we have
$\dot{y}_h(t_h)=\alpha(h)<0$ for all $h$, and by STEP 2 we have
$\dot{y}_h(t)<0$ for all $t\in(0,t_h]$ and $h\in[7/2-\delta,7/2)$.
Moreover any trajectory starts travelling above the magical line with $\dot{y}_h(0)=0$. As long as it stays above it, the particle was only accelatered downwards and thus has a negative velocity in y-direction. Thus for any $h$, there exists an
$\epsilon_h\in(0,t_h)$ with $\dot{y}_h(\epsilon_h)<0$, where we can choose $\epsilon_h$ to depend smoothly on $h$. 
    
Now we argue again by contradiction, assuming that there exists
$h_1\in(0,7/2-\delta)$ and $t_1 \in (0,t_{h_1})$ such that
$\dot{y}_{h_1}(t_1)=0$.
%We may assume that $h_0$ is the greatest such $h$ and $t_0$ is the least such $t$.
Consider the smooth map
$$
   F \colon (0,7/2-\delta]\times [0,1] \to \R,\qquad
    (h,s)\mapsto -\dot{y}_h\bigl((1-s)\cdot
     \epsilon_h +s\cdot t_h\bigr).
$$ 
By assumption it satisfies $F(7/2-\delta,s)>0$ for all
$s\in[0,1]$, $F(h,0)>0$ and $F(h,1)>0$ for all $h\in(0,7/2-\delta)$,
and $F(h_1,s_1)=0$ for some $(h_1,s_1)$. Set
$$
   h_0 := \sup\{h\in(0,7/2-\delta]\mid \text{there exists
     }s\in[0,1]\text{ with }F(h,s)=0\}.
$$
Then $h_0\in(0,7/2-\delta)$, and by continuity of $F$ there exists
$s_0\in(0,1)$ with $F(h_0,s_0)=0$. If $\frac{\p F}{\p s}(h_0,s_0)\neq
0$, then by the implicit function theorem, $F^{-1}(0)$ would near
$(h_0,s_0)$ be a graph $s=s(h)$, contradicting the maximality of $h_0$. 
Hence for each such $s_0$ we must have $\frac{\p F}{\p s}(h_0,s_0)=0$. 
This translates back to the existence of a point $(h_0,t_0)$ with
$\dot{y}_{h_0}(t_0)=0$ and $\ddot{y}_{h_0}(t_0)=0$. Moreover, we can
assume that $h_0$ is maximal with this property, and we can
choose $t_0$ to be minimal given $h_0$. 

But from Chapter~\ref{setup} we know that the only points at which
$\ddot y=0$ lie on the magical line $\sqrt{3}y=|x|$. Thus the point
$(x_{h_0}(t_0),y_{h_0}(t_0))$ must lie on the magical line, and it
must pass this line horizontally because $\dot{y}_{h_0}(t_0)=0$. But
this is impossible since we started on the $y$-axis with horizontal
velocity $\dot{y}_{h_0}(0)=0$ and the force field in the region
between the $y$-axis and the magical line satisfies $\ddot{y}<0$ (the
solution curve $(x_{h_0}(t),y_{h_0}(t))$ for times $t\in (0,t_0)$ must
be completely contained in that region due to the fact that
$\dot{x}(t)>0$ for $t<t_h$ and $t_0$ is the first time at which the
solution curve hits the magical line) so that after any finite
positive time the velocity vector has a negative $y$-component. This
contradiction proves STEP 3. 
\medskip

\subsection{STEP 4: The conclusion of STEP 3 contradicts Proposition ~\ref{PropLimitorbit}.}
Now we will consider the Langmuir problem in the limit $h\to 0$ to
obtain a contradiction. For this, we will change our point of
view. Rather than considering the Langmuir problems at fixed energy
$-1$ and heights $h\searrow 0$, we will consider the Langmuir problems at
{\em fixed height $1$ and energies $-h\nearrow 0$}. By the rescaling
argument in Section~\ref{scale}, these problems differ only in their
time parametrization. By a slight abuse of notation, in this step 
we denote by $q_h=(x_h,y_h)$ the solution to the Langmuir problem at
height $1$ and energy $-h\in(-7/2,0]$. Note that we include the case
$h=0$ which corresponds to the solution $q_0=(x_0,y_0)$ to the
Langmuir problem at height $1$ and energy $0$ considered in
Section~\ref{zero}. For $h\in(0,7/2)$ we denote by $\tau_h>0$ the
first time at which $\dot x_h(\tau_h)=0$ (which differs from the time
$t_h$ in the preceding steps). By STEP 3 we have $\dot{y}_h(t)<0$ for
all $t\in(0,\tau_h]$ and $h\in(0,7/2)$.  

We first claim that $\tau_h\to\infty$ as $h\to 0$. Otherwise there
would exist a sequence $h_n\to 0$ such that $\tau_{h_n}$ converges to
a finite limit $\tau_0\in[0,\infty)$. By continuity, this would imply
$\dot x_0(\tau_0)=0$ and $\dot y_0(\tau_0)\leq 0$. Since $\dot
x_0(0)>0$, we must have $\tau_0>0$. But then in polar coordinates at
time $\tau_0$ we would have $r_0\dot r_0 = x_0\dot x_0+y_0\dot y_0\leq
0$, contradicting Proposition~\ref{PropLimitorbit}. This proves the claim.

From the claim and $\dot{y}_h(t)<0$ for all $t\in(0,\tau_h]$ and
$h\in(0,7/2)$ we deduce $\dot y_0(t)\leq 0$ for all $t\geq 0$. In view
of the initial condition $y_0(0)=1$, this yields $y_0(t)\leq 1$ for
all $t\geq 0$. Moreover, combined with Proposition~\ref{PropLimitorbit} 
it implies $\dot x_0(t)>0$ for all $t>0$. 

If $x_0$ were bounded, then because $y_0$ is also bounded the argument
in STEP 1 would show that $\dot x_0$ must vanish at some positive
time, which it does not. Thus $x_0$ is unbounded. Since $y_0(t)\leq 1$
for all $t$, this implies that the orbit $q_0=(x_0,y_0)$ leaves the
energy zero Hill's region $\{y\geq \frac{1}{\sqrt{63}}|x|\}$ after
some finite time, which is impossible. 
%To be precise, at some finite time $t_1$ we have to pass say $x_1=x(t_1)=8\geq \sqrt{63}$ but the allowed corresponding $y_1$ exceeds 1. This last contradiction shows that STEP 2 prohibits the existence of a \emph{limit orbit}.
%Now by proposition \ref{PropLimitorbit} there exists a \emph{limit orbit} for the Langmuir problem and this contradicts the assertion from STEP 2 that there cannot exist a \emph{limit orbit} if we assume the non-existence of a Langmuir orbit. This is the desired contradiction which proves Theorem~\ref{langmuir}.              
This final contradiction shows that the original assumption was false
and there exists a Langmuir orbit, which proves Theorem~\ref{langmuir}.

\section{Numeric visualisation and outlook}
After the existence of Langmuir's periodic orbit has been established
in the previous sections, we want to give some numerical evidence of
where to find it and where one might look for further periodic
orbits. The figures below show some trajectories of the above
described \emph{Langmuir problem} at energy $E=-1$. The starting
height $h$ is indicated below each picture. We see that for $h=1.398$
we recover Langmuir's periodic orbit.\\ 
\begin{figure}[h!]
 \centering
 \includegraphics[width=8cm,bb=14 14 447 306]{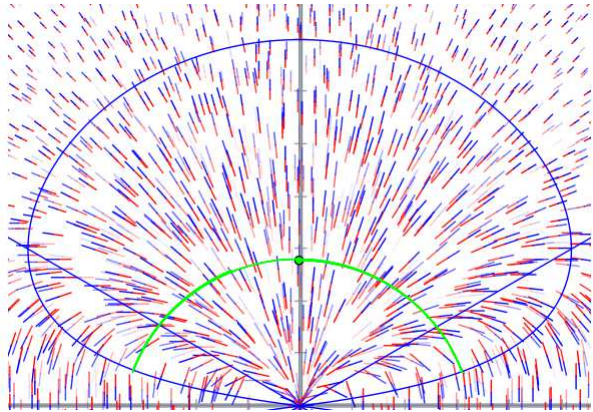}
 % h1398.eps: 1804x1216 px, 300dpi, 15.27x10.30 cm, bb=14 14 447 306
 \caption{The Langmuir orbit $h=1.398$}
\end{figure}

We provide numerical evidence that at least one other periodic orbit
exists in the Langmuir problem. To this end, compare the numerical
sketches for $h=0.3$ and $h=0.8$ shown below. Whereas for $h=0.3$
after bouncing forth and back the end (displayed in red) is reflected
below the part of the trajectory before the last reflection, the end
(also displayed in red) for $h=0.8$ is reflected above that last part
of the trajectory after bouncing forth and back. Continuously varying
the height between $h=0.3$ and $h=0.8$ will eventually give a
trajectory which touches the zero-velocity curve and falls back on the
trajectory it came in. This gives rise to another periodic orbit.   

\begin{minipage}{7cm}

\includegraphics[width=7cm,bb=14 14 312 157]{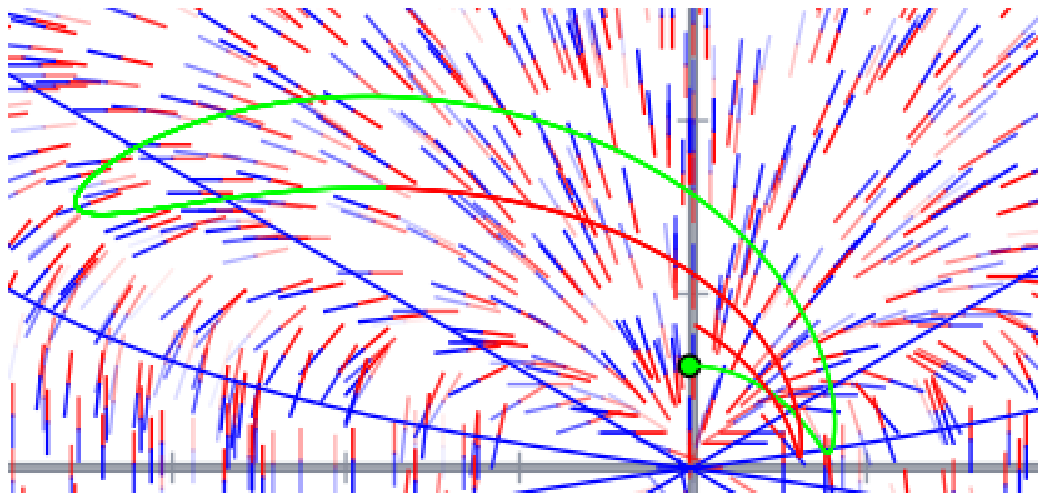}
 % h08.eps: 1241x595 px, 300dpi, 10.51x5.04 cm, bb=14 14 312 157
 \vspace{0.1cm}
 \\
 $h=0.3$
\end{minipage}
\begin{minipage}{7cm}

 \includegraphics[width=7cm,bb=14 14 147 94]{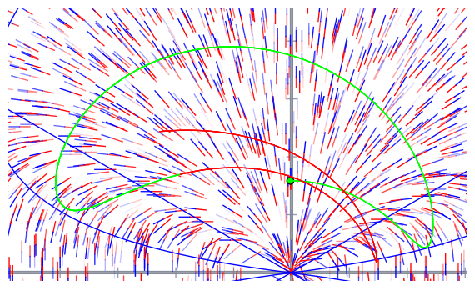}
 % h03.eps: 554x333 px, 300dpi, 4.69x2.82 cm, bb=14 14 147 94
 \vspace{0.1cm}
 
 $h=0.8$

\end{minipage}

The figures were created with \emph{Cinderella 2.8} using a third order Runge-Kutta method.
The green spot indicates where we started on the y-axis.

\end{document}